\documentclass{amsart}
\usepackage{amscd,amssymb,graphicx,epsfig}
\usepackage{array}

\title[Compression of Finite Groups Actions]{Compression of Finite Group Actions and  
Covariant Dimension, II}

\author{Hanspeter Kraft, Roland L\"otscher and Gerald W. Schwarz}
 
\address{Hanspeter Kraft, Roland L\"otscher\newline
\indent Mathematisches Institut der
Universit\"at Basel,\newline
    \indent Rheinsprung 21, CH-4051 Basel, Switzerland}
\email{Hanspeter.Kraft@unibas.ch, Roland.Loetscher@unibas.ch}

\address{Gerald W. Schwarz \newline
    \indent Department of Mathematics\newline
    \indent Brandeis University\newline
    \indent PO Box 549110\newline
    \indent Waltham, MA 02454-9110}
\email{schwarz@brandeis.edu}
\date{May 2008}

\thanks{The first two authors are partially supported by the Swiss National Science
    Foundation (Schweizerischer National\-fonds),
   and the third author by  NSA Grant H98230-06-1-0023}
 \newtheorem{thm}[subsection]{Theorem}
\newtheorem*{thm*}{Theorem}

\newtheorem{prop}[subsection]{Proposition}
\newtheorem{lem}[subsection]{Lemma}

\newtheorem{cor}[subsection]{Corollary}
\newtheorem*{cor*}{Corollary}
\newtheorem*{conj*}{Conjecture}

\theoremstyle{definition}

\newtheorem{exa}[subsection]{Example}

\theoremstyle{remark}
\newtheorem*{rem*}{Remark}
\newtheorem{rem}[subsection]{Remark}
\newtheorem{remarks}[subsection]{Remarks}

\newcommand{\op}{\operatorname}

\newcommand{\name}[1]{\textsc{#1\/}}
\newcommand{\NN}{{\mathbb N}} 
\newcommand{\QQ}{{\mathbb Q}}
\newcommand{\ZZ}{{\mathbb Z}}
\newcommand{\PP}{{\mathbb P}}

\newcommand{\CC}{{\mathbb C}}

\newcommand{\Cst}{{{\mathbb C}^*}}

\newcommand{\OOO}{\mathcal O}

\newcommand{\III}{\mathcal I}

\newcommand{\be}{\begin{enumerate}}
\newcommand{\ee}{\end{enumerate}}

\newcommand{\GL}{\op{GL}}

\newcommand{\PGL}{\op{PGL}}

\newcommand{\rk}{\op{rank}}

\newcommand{\Aut}{\op{Aut}}

\newcommand{\cdim}{\op{covdim}}
\newcommand{\edim}{\op{edim}}

\renewcommand{\Im}{\op{Im}}
\newcommand{\phimax}{{\phi_{\max}}}
\newcommand{\phit}{{{}^{t}\phi}}

\newcommand{\inv}{^{-1}}
\newcommand{\pr}{\op{pr}}

\newcommand{\id}{\op{Id}}

\renewcommand{\phi}{\varphi}
\newcommand{\up}[1]{^{(#1)}}
\newcommand{\pt}{\partial}
 

\begin{document}
\begin{abstract} 
Let $G$ be a finite group and $\phi\colon V\to W$ an equivariant morphism of finite dimensional $G$-modules. We say that $\phi$ is faithful if
$G$ acts faithfully on $\phi(V)$. The covariant dimension of $G$ is the minimum of the dimension of $\overline{\phi(V)}$ taken over all
faithful $\phi$.  In \cite{KS07} we investigated covariant dimension and were able to determine it in many cases. Our techniques largely depended upon finding homogeneous faithful covariants. After publication of \cite{KS07}, the junior author of this article pointed out several gaps in our proofs. Fortunately, this inspired us to find  better techniques, involving multihomogeneous covariants, which have enabled us to extend and complete the results, simplify the proofs and fill the gaps of \cite{KS07}.
 \end{abstract}

 \maketitle
 
 \section{Introduction}
 Our base field is the field $\CC$ of complex numbers.  Let $G$ be a finite group. All $G$-modules that we
consider will be finite dimensional.
A {\it covariant\/} of $G$ is an equivariant morphism $\phi\colon V \to W$ where $V$ and $W$ are
 $G$-modules. The {\it dimension\/} of $\phi$ is defined to be the dimension of the image of $\phi$:
$$
\dim \phi := \dim\overline{\phi(V)}.
$$
The covariant $\phi$ is   {\it faithful\/} if the group $G$ acts faithfully on the image $\phi(V)$. Equivalently, there is  a point $w\in\phi(V)$ with trivial isotropy group $G_w$.
The {\it covariant dimension\/} $\cdim G$ of $G$ is defined to be the minimum of $\dim \phi$ where $\phi\colon V \to W$ runs  over all faithful covariants of $G$. If $\dim\phi=\cdim G$ we say that $\phi$ is a  {\it minimal covariant\/}. In \cite[Proposition 2.1]{KS07} we show that there is a minimal covariant $\phi\colon V\to W$ if $V$ and $W$ are faithful. In particular, if $V$ is a faithful $G$-module, then there is a minimal faithful covariant $\phi\colon V\to V$.  

Suppose that $\phi\colon V\to W$ is a {\it rational map\/} which is $G$-equivariant. We call $\phi$ a {\it rational covariant}. Then one can define  the notion of $\phi$ being faithful and the dimension of 
$\phi$ as in the case of ordinary covariants. The {\it essential dimension $\edim G$\/} of $G$ is the minimum dimension of all its faithful rational covariants. It is easy to see that 
$$
\edim G\leq\cdim G\leq \edim G+1
$$ 
(see \cite[Proposition 2.2]{KS07} or the proof of Theorem \ref{multihomograt.thm} below).

First we show that there are always {\it multihomogeneous\/} minimal covariants and then obtain the exact relation between covariant and essential dimension (Theorem~\ref{covessdim.thm}). In certain cases we are able to describe the image of a covariant (Proposition~\ref{phimax.prop}) and deduce that for a {\it faithful\/} group $G$  (i.e., $G$ admits an irreducible faithful representation) we have
$\cdim (G\times\ZZ/p\ZZ)=\cdim G+1$ if and only if the prime $p$ divides the order $|Z(G)|$ of the center of $G$. This completes the analysis of \cite[\S 5--6]{KS07}. In the process we repair the proofs of Corollaries 6.1 and 6.2 of \cite{KS07}. They are supposed to be corollaries of Proposition 6.1, but the hypotheses of the proposition are not fulfilled. In  section \ref{examples.section} we give some examples of covariant dimensions of  groups,   in part generalizing \cite[Proposition 6.2]{KS07}. In sections \ref{gaschutz.section} and \ref{covdim2.section}  we repair two proofs, one concerning a characterization of faithful groups and their subgroups, and one about the classification of non-faithful groups of covariant dimension 2. In section \ref{errata.section} we list some minor errata from \cite{KS07}.

\section{Multihomogeneous Covariants}\label{multi.section}
 
Let $V=\oplus_{i=1}^n V_i$ and $W=\oplus_{j=1}^m W_j$ be direct sums of vector spaces and   let $\phi=(\phi_1,\dots,\phi_m)\colon V\to W$ be a morphism where none of the $\phi_j$ are zero. We say that $\phi$ is {\it multihomogeneous of degree $A=(\alpha_{ji}) \in M_{m\times n}(\ZZ)$\/} if, for an indeterminate $s$, we have
$$
\phi_j(v_1,\ldots,sv_i,\ldots,v_n) = s^{\alpha_{ji}}\phi_j(v_1,\ldots,v_n) \text{ for all }j=1,\ldots,m, i=1,\ldots,n.
$$
Whenever we consider the degree matrix $A$ of some $\phi$, we are always tacitly assuming that $\phi_j\neq 0$ for all $j$.

We now give a way to pass from a general $\phi$ to the multihomogeneous case.
In general, for indeterminates $s_1,\dots,s_n$, we have $\phi_j(s_1v_1,\dots,s_nv_n)=\sum_\alpha\phi_j\up\alpha s^\alpha$ for each $j$, where $\alpha=(\alpha_1,\dots,\alpha_n)\in\NN^n$ and  $s^\alpha=s_1^{\alpha_1}\dots s_n^{\alpha_n}$. If $\beta\in\NN^n$, let $\alpha\cdot\beta$ denote the usual inner product. Set $h_j=\max\{\alpha\cdot \beta\mid \phi_j\up\alpha\neq 0\}$, $j=1,\dots,m$. For $r\in\NN$ set $\phi_j\up r=\sum_{\alpha\cdot \beta=r} \phi_j\up\alpha$. Now we fix a $\beta$ such that, for each $r\in\NN$ and each $j$, the function $\phi_j\up r$ is zero or consists of one nonzero  
$\phi_j\up\alpha$. Then $\phi_{\max}:=(\phi_1\up {h_1},\dots,\phi_m\up{h_m})$ is multihomogeneous.
Note that the $h_j$ and so $\phi_{\max}$ depend upon our choice of $\beta$.

\begin{remarks}
\begin{enumerate}
\item 
If the $V_i$ and $W_j$ are $G$-modules and $\phi$ is equivariant, so are all the $\phi_j\up\alpha$ and $\phi_{\max}$. Note that no entry in $\phi_{\max}$ is zero since the same is true of $\phi$.
\item
If $\phi\colon V \to W$ is multihomogeneous of degree $A=(\alpha_{ji})$ and $\psi\colon W \to U = \bigoplus_{k=1}^\ell U_\ell$ is multihomogeneous of degree $B = (\beta_{kj})$ and all components of $\psi\circ\phi$ are non-zero, then the composition $\psi\circ\phi\colon V \to U$ is multihomogeneous of degree $B A$. 
\end{enumerate}
\end{remarks}

Concerning $\phi_{\max}$ there is the following main result.
 
\begin{lem}\label{phimax.lem}
Let $\phi\colon V\to W$ be a morphism as above. Then $\dim\phi_{\max}\leq\dim\phi$.
\end{lem}
 
 \begin{proof}
 Let $X$ denote $\overline{\Im\phi}$. We have an action $\lambda$ of $\CC^*$ on $W$ where $\lambda(t)(w)=(t^{h_1}w_1,\dots,t^{h_m}w_m)$ for $w\in W$ and $t\in\CC^*$.  We also have an action $\mu$ of $\CC^*$ on $V$ by $\mu(t)(v_1,\dots,v_n)=(t^{\beta_1}v_1,\dots,t^{\beta_n}v_n)$ where $t\in\CC^*$ and $v\in V$. Let ${}^t\phi(v)$ denote $\lambda(t)(\phi(\mu(t\inv)(v)))$ for $t\in\CC^*$ and $v\in V$. Then ${}^t\phi(v)=\phi_{\max}(v)+t\psi(t,v)$ for some morphism $\psi\colon\CC\times V\to W$. Moreover, for $t\neq 0$, the closure of the image of ${}^t\phi$ is $\lambda(t)$ applied to $X$. Thus the maximal rank of the Jacobian matrix of ${}^t\phi\colon V\to W$ is $\dim X$ for all $t\in\CC^*$. It follows that the rank cannot go up when we specialize $t$ to be zero. Hence $\dim\phi_{\max}\leq\dim\phi$.
 \end{proof}
 
 \begin{rem}\label{phimax.rem} 
 With the notation of the proof above consider the  morphism
 $$
 \Phi\colon \CC\times V \to \CC\times W, \quad(t,v)\mapsto (t,{}^{t}\phi(v))
 $$
 where ${}^{0}\phi:=\phimax$, and let $Y$ denote $\Im\Phi$.
Clearly, we have  $Y\cap(\{t\}\times W) = \{t\}\times\Im\phit$ and so $Y\cap(\{0\}\times W) = \{0\}\times \Im\phimax$.
Denote by $\bar Y$ the closure of $Y$ in $\CC\times W$ and by $p\colon \bar Y \to \CC$ the morphism induced by the projection $\CC\times W \to \CC$. Then
 $$
 \bar Y \cap (\Cst\times W) = \bigcup_{t\neq 0} \{t\}\times \lambda(t)X
 $$
where $X = \overline{\Im\phi}$, because the right hand side is closed in $\Cst \times W$. As a consequence, we get 
$$
\bar Y = \overline{\Phi(\Cst\times W)},
$$
hence ${p}^{-1}(t) = \{t\}\times \overline{\Im\phit}$ for $t\neq 0$ and $p^{-1}(0) \supset \{0\}\times\overline{\Im\phimax}$.
Since $\bar Y$ is irreducible, this again proves Lemma~\ref{phimax.lem} and shows in addition that $\{0\}\times \overline{\Im\phimax}$ is an irreducible component of $p^{-1}(0)$ in case $\dim\phimax = \dim\phi$. This last remark has the following application which will be used later (see Theorem~\ref{multihomograt.thm}). {\it If\/ $\Im\phi$ is stable under scalar multiplication, then so is $\Im\phimax$, provided that $\dim\phimax = \dim\phi$.}
\newline
(In fact, if $\Im\phi$ is stable under scalar multiplication, then so is $\Im{}^{t}\phi$ for all $t\neq 0$ which implies that $\bar Y$ is stable under the $\Cst$-action $\lambda\cdot(t,w):=(t,\lambda w)$ on $\CC\times W$. It follows that $p^{-1}(0)$ is $\Cst$-stable, as well as all its irreducible components.)
\end{rem}
 
 \begin{thm} \label{multihomog.thm}
 Let $G$ be a finite group and let $V=\oplus_{i=1}^n V_i$ and $W=\oplus_{j=1}^m W_j$ be faithful representations where the $V_i$ and $W_j$ are irreducible submodules. Then there is a minimal regular multihomogeneous covariant $\phi\colon V \to W$   all of whose components   are nonzero.
\end{thm}
\begin{proof} Let $\phi\colon V \to W$ be a minimal covariant. We can always arrange that for given $v\in V$ and $w \in W$, both with trivial stabilizer in $G$, we have $\phi(v) = w$ \cite[Proposition 2.1]{KS07}. In particular, we can assume that all components of $\phi$ are nonzero. Then $\phi_{\max}\colon V \to W$ is a multihomogeneous covariant, $\dim\phi_{\max} \leq \dim \phi$ and $\phi_{\max}$ is faithful since all its components are non-zero \cite[Lemma 4.1]{KS07}.
\end{proof}
 
 \begin{cor} \label{multihomog.cor}
 Let $V_i$ be a faithful irreducible representation of the group $G_i$, $i=1,\dots,n$. Then $V=\oplus_{i=1}^n V_i$ is a faithful representation of $G:=G_1\times\dots \times G_n$, and there is a minimal multihomogeneous covariant $\phi\colon V\to V$.
 \end{cor}

We want to prove similar results for a {\it rational\/} covariant $\psi\colon V \to W$. It is obvious how to extend the definitions of   {\it minimal\/} and {\it multihomogeneous of degree $A$\/} to rational covariants where in this case the matrix $A$ might contain negative entries. 

\begin{thm} \label{multihomograt.thm}
Let $G$ be a finite group and let $V=\oplus_{i=1}^n V_i$ and $W=\oplus_{j=1}^m W_j$ be faithful representations where the $V_i$ and $W_j$ are irreducible submodules. Then there is a minimal rational multihomogeneous covariant $\psi\colon V \to W$   all of whose components   are non-zero and which is of the form $\psi = h^{-1}\phi$ where $h$ is a multihomogeneous invariant and $\phi\colon V \to W$ a multihomogeneous minimal regular covariant.
\end{thm}

\begin{proof} 
Let $\psi\colon V \to W$ be a minimal rational covariant. We can assume that all components of $\psi$ are nonzero. 
There is a nonzero invariant $f\in\OOO(V)^G$ such that $f\psi$ is regular. Define the regular covariant
$$
\phi:=(f\psi,f)\colon V \to W\oplus \CC, \quad v\mapsto (f\psi(v),f(v))
$$
which is faithful since $\psi$ is. Moreover, either $\dim\phi=\dim\psi$ or $\dim\phi = \dim\psi+1$, where the second case takes place if and only if 
$\overline{\phi(V)}$ is stable under scalar multiplication with $\Cst$. This follows from the fact  that the composition of rational maps  $V \to W\oplus\CC \to \PP(W\oplus\CC) \to W$ is   $\psi$. 

As above we obtain a multihomogeneous covariant $\phi_{\max}\colon V \to W\oplus\CC$ which has the form $\phi_{\max} = (\phi_1,\ldots,\phi_m,h)$. Now define the multihomogeneous rational covariant
$$
\psi_{\max} := (\frac{\phi_1}{h},\ldots, \frac{\phi_m}{h}) \colon V \to W
$$
which is again faithful. Moreover, $\dim\psi_{\max} \leq \dim\phi_{\max} \leq \dim \phi$. So if  $\dim\phi = \dim\psi$ then $\psi_{\max}$ is a minimal multihomogeneous rational covariant and we are done.

Now assume that  $\dim\phi = \dim\psi + 1$   so that $\overline{\phi(V)}$ is $\Cst$-stable. If $\phi$ is not minimal then there is a minimal homogeneous regular covariant $\tilde\phi$ of dimension $\leq\dim\psi$ and we are again done. 
Therefore we can assume that $\phi$ is minimal, hence $\dim\phimax = \dim \phi$. Since  $\overline{{}^t\phi(V)}$ is $\Cst$-stable for all $t\neq 0$ it follows from Remark~\ref{phimax.rem}
that $\overline{\phi_{\max}(V)}$ is $\Cst$-stable, too, and so
$$
\dim\psi_{\max} \leq \dim \phi_{\max}-1\leq \dim \phi -1 =  \dim\psi.
$$ 
Hence, $\psi_{\text{max}}$ is a minimal multihomogeneous rational covariant.
\end{proof}

\section{Covariant dimension and essential dimension}\label{cdim-edim.section}
In this section we extend \cite[Corollary 4.2]{KS07} to arbitrary groups and give the exact relation between covariant and essential dimension of finite groups.
\begin{thm}\label{covessdim.thm}
Let $G$ be a non-trivial finite group. Then $\cdim G = \edim G$ if and only if $G$ has a non-trivial center.
\end{thm}
The proof is given in Corollary~\ref{edimcdim.cor} and Proposition~\ref{edimcdim.prop} below. We need  some preparation. In this section we have faithful representations $V=\bigoplus_{i=1}^{n}V_{i}$ and $W=\bigoplus_{j=1}^{m}W_{j}$ where the $V_i$ and $W_j$ are irreducible submodules.  We have a natural action of the tori $\Cst^n$ on $V$ and $\Cst^m$ on $W$. These actions are free on the open sets $V':=\{v=(v_{1},\ldots,v_{n})\mid v_{i}\neq 0\text{ for all }i\} \subset V$ and 
$W'\subset W$ defined similarly. If $\phi\colon V \to W$ is multihomogeneous of degree $A =(\alpha_{ji})$ then $\phi$ is equivariant with respect to the homomorphism 
$$
T(A)\colon \Cst^n \to \Cst^m, \qquad s=(s_1,\ldots,s_n)\mapsto (s^{\alpha_1},s^{\alpha_2},\ldots,s^{\alpha_m})
$$
where $\alpha_j := (\alpha_{j1},\alpha_{j2},\ldots,\alpha_{jn})$ and $s^{\alpha_j} =  s_1^{\alpha_{j1}}s_2^{\alpha_{j2}}\cdots s_n^{\alpha_{jn}}$, as before. This implies that the (closure of the) image of $\phi$ is stable under the subtorus $\Im T(A) \subset \Cst^m$. 
The actions of $G$ and $\Cst^n$ commute and so, considered as subgroups of $\GL(V)$,  we have $\Cst^n \cap G = Z(G)$.
\begin{rem}\label{invarcovar.rem}
Let $\phi\colon V \to W$ be a multihomogeneous covariant of degree $A$. If $\mu \in A\QQ^{n}\cap\ZZ^{m}$, then $\phi(V)\subset W$ is stable under the $\Cst$-action $\rho(t)(w_{1},\ldots,w_{m}) := (t^{\mu_{1}}w_{1},\ldots,t^{\mu_{m}}w_{m})$. It follows that for any invariant $f\in\OOO(V)^{G}$ the morphism
\begin{equation}\label{eqn1}
\tilde\phi\colon v=(v_{1},\ldots,v_{n}) \mapsto (f(v)^{\mu_{1}}\phi_{1}(v),\ldots, f(v)^{\mu_{m}}\phi_{m}(v))
\end{equation}
is a covariant with $\tilde\phi(V)\subset\phi(V)$,  hence $\dim\tilde\phi\leq\dim\phi$. Moreover,  if $\phi$ is faithful and $f$ multihomogeneous, then $\tilde\phi$ is faithful and multihomogeneous of degree
$\tilde A :=  \mu   \deg f + A$, i.e., $\tilde \alpha_{ji} = \mu_{j}\deg_{V_{i}}f + \alpha_{ji}$. 

This has the following application which will be used later in the proof of Corollary~\ref{timesZp.cor}: {\it Let $p$ be a prime which does not divide the order of the center of $G$. Then there is a minimal multihomogeneous covariant $\phi\colon V \to V$ of degree $A\not\equiv 0 \mod p$.} 
\newline
(Start with a minimal multihomogeneous covariant $\phi\colon V \to V$ of degree $A$ and assume that $A \equiv 0 \mod p$. We can choose a $\mu\in A\QQ^{n}\cap\ZZ^{m}$ such that $\mu_{j_{0}}\not\equiv 0 \mod p$ for at least one $j_{0}$. Moreover, there is a multihomogeneous invariant $f$ of total degree $\not\equiv 0 \mod p$ (see \cite[Lemma~4.3]{KS07}). But then $\mu  \deg f \not\equiv 0 \mod p$, and so the covariant $\tilde\phi$ given in (\ref{eqn1}) is minimal and has degree $\mu  \deg f + A \not \equiv 0 \mod p$.)
\end{rem}

For the next results we need some preparation. Let $\phi\colon V \to W$ be a multihomogeneous faithful covariant of degree $A = (\alpha_{ji})$ where all components $\phi_{j}$ are non-zero. Define $W':=\{(w_{1},\ldots,w_{m})\in W\mid
w_{i}\neq 0 \text{ for all }i\}=\prod_{j=1}^{m}(W_{j}\setminus\{0\})$. The group $\Cst^{m}$ acts freely on $W'$ and $W' \to \prod_{j=1}^{m}\PP(W_{j})$ is the geometric quotient. Let $X:=\overline{\phi(V)}$ and $\PP(X) \subset \prod_{j=1}^{m}\PP(W_{j})$ the image of $X$, and set $X' := X \cap W'$. Finally, denote by $S\subset \Cst^{m}$ the image of the homomorphism $T(A)\colon \Cst^{n}\to \Cst^{m}$. Then we have the following.
\begin{lem}\label{centerG.lem} 
\be
\item $\dim \PP(X) \leq \dim X -\dim S \leq \dim X -\rk Z(G)$.
\item The kernel of the action of $G$ on $\PP(X)$ is equal to $Z(G)$.
\ee
\end{lem}
\begin{proof}
We may regard $G$ as a subgroup of $\prod_{i=1}^{m}\GL(V_{i})$ and of  $\prod_{j=1}^{m}\GL(W_{j})$, and so $Z(G) = G\cap \Cst^{n}$ and $Z(G) = G \cap \Cst^{m}$. 

(1) The first inequality is clear because $X$ is stable under $S$. For the second we remark that $Z(G) \subset S$ since $\phi$ is $G$-equivariant and so $T(A)z = z$ for all $z\in Z(G)$.

(2) Let $g\in G$ act trivially on $\PP(X)$. Then every $x\in X_{j}:=\pr_{W_{j}}(X)$ is an eigenvector of $g|_{W{j}}$. But $X_{j}$ is irreducible and therefore contained in a fixed eigenspace of $g$ on $W_{j}$. Since $W_{j}$ is a simple $G$-module this implies that $g|_{W_{j}}$ is a scalar.
\end{proof}

\begin{prop}\label{rankA.prop}
Let $\phi\colon V \to W$ be a multihomogeneous faithful covariant of degree $A = (\alpha_{ji})$ where all components $\phi_{j}$ are non-zero. Assume that $G$ has a trivial center. Then 
$$
\edim G \leq \dim\phi - \rk A \text{ and }
\cdim G \leq \dim\phi - \rk A +1.
$$ 
In particular, if $\phi$ is a minimal regular covariant, then $\rk A =1$, and if $\phi$ is a minimal rational covariant, then $A = 0$.
\end{prop} 
\begin{proof} Let $X:=\overline{\phi(V)}$, let  $\PP(X) \subset \prod_{j=1}^{m}\PP(W_{j})$ denote the image of $X$  and set $X' := X \cap W'$. Finally, let $S$ denote the image of  $T(A)\colon \Cst^{n}\to \Cst^{m}$. The torus $S$ has dimension $ \rk A$ and acts generically freely on $X:=\overline{\phi(V)}$ since all components of $\phi$ are non-zero. Composing  $\phi$ with the projection $p\colon W \to \PP(W_1) \times \cdots\times \PP(W_m)$ we obtain a rational $G$-equivariant map $\phi'\colon V \to \PP(W_1) \times \cdots\times \PP(W_m)$ such that $\overline{\phi'(V)} = \PP(X)$. Since $Z(G)$ is trivial, $G$ acts faithfully on $\PP(X)$, and $\dim \PP(X) \leq \dim X - \dim S$, by Lemma~\ref{centerG.lem}. Thus $p\circ\phi'$ is a rational faithful covariant of dimension $\leq \dim X - \rk A$, proving the first claim. The second follows since $\cdim G \leq \edim G + 1$.
\end{proof}

\begin{cor}\label{edimcdim.cor}
If $G$ is a (non-trivial) group with trivial center, then 
$$
\cdim G = \edim G + 1.
$$
\end{cor}
\begin{proof} Let $\phi\colon V \to V$ be a minimal multihomogeneous regular covariant of degree $A$. By  Proposition~\ref{rankA.prop}, $\rk A = 1$ and $\phi$ is not minimal as a multihomogeneous rational covariant. Hence $\edim G < \dim \phi = \cdim G$ and the claim follows.
\end{proof}
\begin{prop}\label{edimcdim.prop}
If $G$ has a non-trivial center, then $\cdim G = \edim G$.
\end{prop}
\begin{proof} Let $\psi\colon V \to V$ be a multihomogeneous minimal rational covariant of degree $A = (\alpha_{ji})$ which is of the form $h^{-1}\phi$ where $h\in\OOO(V)^G$ is a multihomogeneous invariant and $\phi\colon V \to V$ a multihomogeneous regular minimal covariant (Theorem~\ref{multihomograt.thm}). 

(a) If there is a $\beta\in\ZZ^n$ such that all entries of $\gamma:=A\beta$ are $>0$, then the covariant 
$\phi:=(h^{\gamma_1}\psi_1,\ldots,h^{\gamma_n}\psi_n)\colon V \to V$ is regular and faithful. 
Moreover, $\overline{\phi(V)} \subset \overline{\psi(V)}$ because the latter is stable under $T(A)(\Cst^n)$. Hence $\cdim G \leq \dim\phi \leq \dim\psi = \edim G$ and we are done.

(b) In general, $A \neq 0$, since otherwise the center of $G$ would act trivially on the image $\psi(V)$. If   $\alpha_{j_0i_0} \neq 0$,   choose a homogeneous invariant $f\in\OOO(V_{j_0})\subset\OOO(V)$ which does not vanish on $\psi(V)$. For any $r\in\ZZ$ the composition $\psi':=(f^r \cdot \id) \circ \psi$ is still faithful and rational, and $\dim\psi' \leq \dim \psi$. Moreover, we get $\psi'_{j}(v) = f^{r}(\psi_{j_0}(v))\cdot \psi_j(v)$. Therefore the degree of $\psi'_j$ in $V_{i_0}$ is $r \cdot\deg f\cdot \alpha_{j_0i_0} + \alpha_{ji_0}$ for $j=1,\ldots,n$. Hence, for a suitable $r$, all these degrees are $>0$, and we are in case (a) with $\beta := e_{i_0}$.
\end{proof}
In some of our applications we will need the following result.

\begin{cor}\label{centralExtension.cor} Assume that the center $Z(G)$ is cyclic (and non-trivial) and that $Z(G)\cap(G,G) = \{e\}$. If $G/Z(G)$ is faithful, then $G$ is faithful, too, and
$$
\edim G = \cdim G = \cdim G/Z(G) = \edim G/Z(G) + 1.
$$
\end{cor}
\begin{proof}
It easily follows from the assumption $Z(G)\cap(G,G) = \{e\}$ that the center of $G/Z(G)$ is trivial and that every character of $Z(G)$ can be lifted to a character of $G$. Now let $V$ be an irreducible faithful representation of $G/Z(G)$ and let $\phi\colon V \to V$ be a homogeneous minimal covariant. Since $G/Z(G)$ has a trivial center we may assume that the degree of $\phi$ is $\equiv 1 \mod |Z(G)|$ (see Remark~\ref{invarcovar.rem}). If $\chi\colon G \to \Cst$ is a character which is faithful on $Z(G)$ then $V\otimes\chi$ is an irreducible faithful representation of $G$ and $\phi\colon V\otimes\chi \to V\otimes\chi$ is $G$-equivariant and faithful. Hence $\cdim G = \cdim G/Z(G)$. The other two equalities follow with  Proposition~\ref{edimcdim.prop} and Corollary~\ref{edimcdim.cor}.
\end{proof}

\section{The image of a covariant}
In certain cases one can get a handle on the ideal of $\Im\phi$.
 
 \begin{prop}\label{phimax.prop} Let  $V:=\oplus_{i=1}^n V_i$ and let 
 $\phi=(\phi_1,\dots,\phi_n) \colon V \to V$ be a multihomogeneous morphism of degree $A=(\alpha_{ji})$. Assume that $\det A \neq 0$. Then the ideal $\III(\phi(V))$ of the image of $\phi$ is generated by multihomogeneous polynomials.
 \end{prop}

\begin{proof} For $v=(v_1,\dots,v_n) \in V$ we have   
$$
 \phi(s_1v_1,\dots,s_nv_n)=(s^{\alpha_1}\phi_1,\dots,s^{\alpha_n}\phi_n)(v)
 $$
 where  $s^{\alpha_j}=s_1^{\alpha_{j1}}\cdots s_n^{\alpha_{jn}}$.   Choose coordinates in each $V_i$ and let $M$ be a monomial in these coordinates. Let $\beta=\beta(M)$ denote the multidegree of $M$, so we have $M(s_1v_1,\dots,s_nv_n)=s^{\beta}M(v_1,\dots,v_n)$. Then  $M(\phi(s_1v_1,\dots,s_nv_n))$ is $M(\phi(v))$ multiplied by 
  $$
 (s^{\alpha_1},\dots,s^{\alpha_n})^\beta  =s_1^{\beta_1\alpha_{11}+\dots+\beta_n\alpha_{n1}}\cdots s_n^{\beta_1\alpha_{1n} +\dots+ \beta_n\alpha_{nn}} =s^{\beta A} 
 $$
 where $\beta A$ is the matrix product of $\beta$ and $A$. 
If $F\in\III(\phi(V))$, we may write $F=\sum_M  c_M M$ where the $c_M$ are constants and $M$ varies over all monomials in the coordinates of the $V_i$. We have  $F(\phi(s_1v_1,\dots,s_nv_n))=\sum_M c_M s^{\beta(M) A}M(\phi(v))$. Hence,  for any $\gamma\in\NN^n$, we obtain
 $$
 \sum_{\beta(M) A =\gamma} c_M M\in\III(\phi(V)).
 $$
Since $\det A \neq 0$,  for any $\gamma$ there is at most one $\beta$ such that $\beta A =\gamma$. It follows that every sum of the form $\sum_{\beta(M)=\beta} c_M M$ belongs to $\III(\phi(V))$. Thus $\III(\phi(V))$ is generated by multihomogeneous polynomials.
 \end{proof}
   
 \begin{cor}\label{dimphi.cor}
 Suppose that $\phi$ is as above and that there is a $k$, $1 \leq k<n$, such that  $\dim V_{k+1}=\dots=\dim V_n=1$. Then $\dim\phi=\dim(\phi_1,\dots,\phi_k)+(n-k)$.
 \end{cor}
 
 \begin{proof} Since the degree matrix  $A=(\alpha_{ji})$ exists,   no $\phi_j$ is zero. Let $m=\dim V_1+\dots+\dim V_k$. By Proposition~\ref{phimax.prop}  the ideal of $\overline{\phi(V)}$ is generated by functions of the form $F(y_1,\dots,y_m)t_{k+1}^{r_{k+1}}\cdots t_n^{r_n}$ where $F$ is multihomogeneous. Such a function vanishes on $\Im\phi$ if and only if $F(y_1,\dots,y_m)$ vanishes on the image of $(\phi_1,\dots,\phi_k)$. Thus the ideal $\III(\phi(V))$ is generated by functions not involving $t_{k+1},\dots, t_n$. As a consequence, $\overline{\phi(V)}=\overline{(\phi_1,\dots,\phi_k)(V)}\times V_{k+1}\times\dots\times V_n$.
 \end{proof}

In order to apply Proposition~\ref{phimax.prop} and Corollary~\ref{dimphi.cor} we need a version of \cite[Lemma 5.2]{KS07}. 

\begin{cor}\label{alpha.cor}
Let $G=G_1\times\cdots\times G_n$ and $V=V_1\oplus\dots\oplus V_n$ where each $V_i$ is an irreducible representation of $G_i$, $i=1,\dots,n$. Let $\phi\colon V\to V$ be a multihomogeneous covariant of degree $A$   and suppose that the prime $p$ divides $|Z(G_i)|$ for all $i$. Then $\det A \neq 0$, and the ideal $\III(\phi(V))$ is generated by multihomogeneous elements.
\end{cor}

\begin{proof}
Let $\xi$ be a primitive $p$th root of unity. Then we have $\phi_j(v_1,\dots,\xi v_i,\dots,v_n)=\xi^{\alpha_{ji}}\phi_j(v_1,\dots,v_n)$.  There is an element of $G_j$ which acts as $\xi$ on $V_j$ and trivially on $V_i$ if $i\neq j$. Hence $\xi^{\alpha_{ji}}=1$ for $i\neq j$. If $i=j$, one similarly shows that $\xi^{\alpha_{jj}}=\xi$ by equivariance relative to $G_j$. This implies that 
$$
\alpha_{ji}\equiv
\begin{cases} 1 \mod p &\text{for }i=j,  \\
0 \mod p &\text{otherwise,} 
\end{cases}
$$
and so $\det(\alpha_{ij}) \neq 0$. Now apply Proposition~\ref{phimax.prop}.
\end{proof}
 
 We say that $G$ is \emph{faithful\/} if it admits a faithful irreducible representation.
We now get the following  result  which extends Corollaries 6.1 and 6.2 of \cite{KS07}.
 
 \begin{cor}\label{timesZp.cor}
Let $G=G_1\times\cdots\times G_n$ be a product of non-trivial faithful groups and let $p$ be a prime.
\be
\item If $p$ is coprime to $|Z(G)|$, then $\cdim (G\times\ZZ/p) = \cdim G$.
\item If $p$ divides all $|Z(G_i)|$, then $\cdim (G \times (\ZZ/p)^m) = \cdim G + m$.
\ee 
In particular, if $H$ is a non-trivial faithful group and $m\geq 1$, then
$$
\cdim (H \times (\ZZ/p)^m) = 
\begin{cases} \cdim H + m  &\text{if $p$ divides $|Z(H)|$};\\ 
\cdim H + (m-1) &\text{otherwise.} 
\end{cases}
$$
\end{cor}
\begin{proof}
Let $V_i$ be a faithful irreducible representation of $G_i$. Then $V:= V_1\oplus\cdots\oplus V_n$ is a faithful representation of $G$. By  Corollary~\ref{multihomog.cor} there is a minimal multihomogeneous faithful covariant $\phi=(\phi_1,\ldots,\phi_k)\colon V\to V$  of degree $A$.
For any $\delta=(\delta_{1},\ldots,\delta_{n})\in\ZZ^{n}$ there is a linear action of $\ZZ/p$ on $V$ where the generator $\bar 1\in\ZZ/p$ acts by 
$$
v= (v_{1},\ldots,v_{n}) \mapsto (\zeta^{\delta_{1}}v_{1},\ldots,\zeta^{\delta_{n}}v_{n}),\quad \zeta:= e^{\frac{2\pi i}{p}}.
$$
This actions commutes with the $G$-action and defines a $G\times\ZZ/p$-module structure on $V$ which will be denoted by $V_{\delta}$. It follows that for  $\mu = A\delta$ the multihomogeneous map $\phi$ is a $G\times\ZZ/p$-equivariant morphism $\phi\colon V_{\delta}\to V_{\mu}$. If $p$ is coprime to $|Z(G)|$ we can assume that $A \not\equiv 0 \mod p$ (Remark~\ref{invarcovar.rem}).  Then  there is a $\delta$ such that $\mu = A\delta \not\equiv 0 \mod p$ and so $\phi$ is a faithful covariant for the group $G\times\ZZ/p$, proving (1).

Assume now that $p$ divides all $|Z(G_i)|$. There is a minimal multihomogeneous covariant
$\psi\colon V\oplus\CC^m \to V\oplus\CC^m$ for $G\times(\ZZ/p)^m$ where $(\ZZ/p)^m$ acts in the obvious way on $\CC^m$. Clearly, no entry of $\psi$ is zero and by Corollaries~\ref{alpha.cor} and \ref{dimphi.cor},  we get
$\dim\psi=\dim\phi+m$ where $\phi\colon V\oplus\CC^m\to V$ is $\psi$ followed by projection to $V$. Since each component of $\phi $ is nonzero, $\phi$ is faithful for $G$    \cite[Lemma 4.1]{KS07}. Thus $\dim\phi \geq\cdim G$. But clearly, $\cdim (G\times(\ZZ/p)^m)\leq\cdim G+m$, hence we have equality, proving (2).
\end{proof}

As an immediate consequence we get the following result.
\begin{cor}\label{abelian.cor}
Let $G$ be abelian of rank $r$. Then $\cdim G=r$.
\end{cor}

\begin{rem} The corollary  is Theorem 3.1 of \cite{KS07}. 
The proof in    \cite{KS07} uses a lemma whose proof is incorrect. The problem is that the quotient ring $R/pR$ constructed  there may have zero divisors. However, one can give a correct proof of the lemma by paying attention to the powers of the variables that occur in the  determinant $\det (\pt f_i/\pt x_j)$. We omit this proof since the lemma is no longer needed.
\end{rem}

The following strengthens \cite[Proposition 6.1]{KS07}, which in turn then simplifies other proofs in the paper, e.g., the proof of Proposition 6.2.

\begin{cor}\label{Gwithcharacter.cor}
 Let $V=W\oplus \CC_\chi$ be a faithful representation of $G$ where $W$ is  irreducible and $\chi$ is a character of $G$. Let $H$ denote the kernel of $G\to\GL(W)$. Assume that there is a prime $p$ which divides the order of $H$ and such that the following two equivalent conditions hold:
 \be
  \item[(i)] There is a subgroup of $\ker\chi$ acting as scalar multiplication by $\ZZ/p$ on $W$;
\item[(ii)] There is a subgroup of $G$ acting as scalar multiplication by  $\ZZ/p$ on $V$.
 \ee
 Then $\cdim G = \cdim G/H + 1$.
\end{cor}

\begin{proof} It is easy to see that the two conditions are equivalent, because $\chi|_{H}\colon H \to \Cst$ is injective.

Let $(\phi,h)\colon W \oplus \CC_\chi \to W \oplus \CC_\chi$ be a minimal faithful multihomogeneous covariant of degree $\deg(\phi,h) = (\alpha_{ji})$. Since $H$ is nontrivial, $h$ cannot be zero. 
By assumption, $H$ contains a subgroup of order $p$ which is mapped injectively into $\Cst$ by $\chi$. Thus the subgroup acts trivially on $W$ and by scalar multiplication on $\CC_\chi$. Therefore, 
$$
\alpha_{22}\equiv 1 \text{ and } \alpha_{12} \equiv 0 \mod p.
$$
Similarly, condition (i) implies that
$$
\alpha_{11}\equiv 1 \text{ and } \alpha_{21} \equiv 0 \mod p.
$$
Thus $\det(\alpha_{ij}) \neq 0$, and so $\dim (\phi,h) = \dim\phi + 1$ by Corollary~\ref{dimphi.cor}. The equivariant morphism $\phi\colon W \oplus\CC_{\chi} \to W$ factors through the quotient $(W\oplus\CC_{\chi})/H$ which is isomorphic to the $G/H$-module $W\oplus\CC$, and defines a faithful $G/H$-covariant $\bar\phi\colon W\oplus\CC \to W$.  Hence, $\dim \phi\geq\cdim G/H$, and our result follows.
\end{proof}

Now consider the following commutative diagram with exact rows where $\ell>m\geq 0$, $\mu_N \subset \Cst$ denotes the $N$-th roots of unity and $\pi$ is the canonical homomorphism $\xi \mapsto \xi^{p^{\ell-m}}$:
$$
\begin{CD}
1 @>>> K @>>> G @>{\chi}>> \mu_{p^\ell} @>>> 1\\
&& @|  @VVV  @V\pi VV \\
1 @>>> K @>>> G' @>{\chi'}>> \mu_{p^m} @>>> 1\\
\end{CD}
$$
\begin{cor}\label{appl.cor}
In the diagram above 
assume that $G'$ is faithful and that the prime $p$ divides $|Z(G')\cap K|$. Then $\cdim G \geq \cdim G' + 1$.
\end{cor}
\begin{proof} 
Let $\rho\colon G' \to \GL(W)$ be a faithful irreducible representation. Then $V:= W \oplus \CC_\chi$ is a faithful representation of $G$. Fix a $p$-th root of unity $\zeta\in\Cst$ and let $z'\in Z(G')\cap K$ be such that $\rho(z')=\zeta \cdot\id$. We have
$$
G = \{(g',\xi)\in G'\times \mu_{p^\ell} \mid \chi'(g')=\pi(\xi) \}
$$
and so $z :=(z',\zeta)\in Z(G)$ acts as scalar multiplication with $\zeta$ on $V$. Now the claim follows from Corollary~\ref{Gwithcharacter.cor}.
\end{proof}

\section{Some examples}\label{examples.section}

We consider the covariant dimension of some products and semidirect products of groups. We denote by $C_{n}$ a cyclic group of order $n$.

\begin{exa} Consider the group $G:=C_{3}\rtimes C_{4}$ where a generator of $C_{4}$ acts on $C_{3}$ by sending each element to its inverse. Then $Z(G) \subset C_{4}$ is of order 2, $(G,G)=C_{3}$ and $G/Z(G) \simeq S_{3}$. Hence $\edim G = \cdim G = \cdim S_{3} = 2$, by Corollary~\ref{centralExtension.cor}.
\end{exa}

\begin{exa}\label{S3S3.ex} 
Let $H:=S_3\times S_3$. Since $\cdim S_3=2=\edim S_{3}+1$, we have $\cdim H = \edim H + 1 \leq 2\edim S_{3}+1\leq 3$.  We claim that $\cdim H = 3$. Let $G$ denote $H\times (\ZZ/2\ZZ)^2$. By Corollary \ref{timesZp.cor}, $\cdim G=\cdim H+1$. Since $G$ contains a copy of $(\ZZ/2\ZZ)^4$, its covariant dimension is at least $4$, hence it is 4, and so the covariant dimension of $H$ is 3. The same reasoning shows that $\cdim S_3\times S_4=4$ and $\cdim S_4\times S_4=5$.
\end{exa}

\begin{exa} Let $G:= A_{4}\rtimes C_{4}$ where a generator $x$ of $C_{4}$ acts on $A_{4}$ by conjugation with a $4$-cycle $\sigma\in S_{4}$. We get
$$
Z(G) =\langle x^{2}\sigma^{2}\rangle \simeq C_{2}, \ (G,G) = A_{4}, \  G/Z(G) \simeq S_{4}.
$$
Thus $\edim G = \cdim G = \cdim S_{4} = 3$, by Corollary~\ref{centralExtension.cor}. Moreover, $G$ has a $3$-dimensional faithful representation---the standard representation of $A_{4}$ lifts to a faithful representation of $G$---and $G$ contains a subgroup isomorphic to $C_{2}\times C_{2}\times C_{2}$.
\end{exa}

\begin{exa} Let $\sigma\in S_{n}\setminus A_{n}$ be of (even) order $m$ where $n\geq 4$, and consider the group $G := A_{n} \rtimes C_{m}$ where a generator of $C_{m}$ acts on $A_{n}$ by conjugation with $\sigma$. Again, we can apply Corollary~\ref{centralExtension.cor} and get $\edim G = \cdim G = \cdim S_{n}$.
\end{exa}

\begin{exa} Let $G:=(C_{3}\times C_{3})\rtimes (C_{4}\times C_{8})$ where a generator $x$ of $C_{4}$ acts on $C_{3}\times C_{3}$ by sending each element to its inverse, and a generator $y$ of $C_{8}$ by sending the first component to its inverse and leaving the second component invariant. Then $Z(G) = \langle x^{2},y^{2} \rangle \simeq C_{2}\times C_{4}$, $(G,G) = C_{3}\times C_{3}$ and $G/Z(G) \simeq S_{3} \times S_{3}$. Since the center is not cyclic we cannot apply Corollary~\ref{centralExtension.cor} directly, but have to pass through the intermediate group $\bar G := G/\langle x^{2}\rangle$ which has a cyclic center, namely $\langle y^{2}\rangle$. Thus we obtain $\edim \bar G = \cdim \bar G = \cdim \bar G/Z(\bar G) = \cdim S_{3}\times S_{3} = 3$ by Example~\ref{S3S3.ex}. Since $\bar G$ is faithful we can apply Corollary~\ref{Gwithcharacter.cor}: Take $H:=\langle x^{2}\rangle$ and choose for $\chi$ a lift of the character $\bar\chi$ on $Z(G) =\langle x^2,y^2 \rangle$ given by $\bar\chi(x^{2}) = -1$ and $\bar\chi(y^{2})= 1$. We finally get $\edim G = \cdim G = \cdim \bar G + 1 = 4$.
\end{exa}

\begin{exa} In this example we study semi-direct products of cyclic $p$-groups where $p$ is a fixed prime. We generalize  \cite[Proposition 6.2]{KS07} with a simpler proof.  By a recent general theorem due to \name{Karpenko-Merkurjev} \cite{KaM08} this case could be reduced to a question of representation theory. {\it For any finite $p$-group $G$ the essential dimension $\edim G$ equals the minimal dimension of a faithful representation of $G$.} We will only treat some very special cases, without using this result.

 Let  $\ell>1$ be an integer and $\alpha$ an automorphism of $\ZZ/p^\ell$ of order $|\alpha|=p^d$ (hence $ d<\ell$). For any $k\geq d$ we can form the semidirect product $G_p(k,\ell,\alpha):=\ZZ/p^k \ltimes \ZZ/p^\ell$ where the generator $\bar 1$ of $\ZZ/p^k$ induces the automorphism $\alpha$. We first remark that $G_p(k,\ell,\alpha)$ is faithful if and only if $|\alpha|=p^k$. This follows easily from Gasch\"utz's criterion (see Proposition~\ref{prop-Gaschutz} in the next section).

The ``smallest'' non-commutative example is $G_p(1,2,\alpha) = \ZZ/p\ltimes\ZZ/p^2$ where $\alpha\in\Aut(\ZZ/p^2)$ has order $p$, e.g., $\alpha(\bar 1) = \overline{p+1}$. It is also clear that every non-commutative $G(k,\ell,\alpha)$ contains a subgroup of the form $G(k,\ell',\alpha')$ where $\ell'\leq\ell$ and $\alpha'$  has order $p$. Thus we obtain the following commutative diagram with exact rows:
$$
\begin{CD}
1 @>>> \ZZ/p^\ell @>>> G_p(k,\ell,\alpha) @>>> \ZZ/{p^k} @>>> 1\\
&& @A\subseteq AA  @A\subseteq AA  @ | \\
1 @>>> \ZZ/p^{\ell'} @>>> G_p(k,\ell',\alpha') @>>> \ZZ/{p^k} @>>> 1\\
&& @|  @VVV  @V\pi VV \\
1 @>>> \ZZ/p^{\ell'} @>>> G_p(1,\ell',\alpha') @>>> \ZZ/{p} @>>> 1\\
\end{CD}
$$
If we apply Corollary~\ref{appl.cor} to the bottom half of the diagram and then use that $G_p(k,\ell',\alpha')$ is a subgroup of  $G_p(k,\ell,\alpha)$ we find 
$$
\cdim G_p(k,\ell,\alpha) \geq \cdim G_p(1,\ell',\alpha') + 1
$$
if $\alpha$ is non-trivial and $k>1$ since in this case $G_p(1,\ell',\alpha')$ is faithful.

For $p=2$ every $G_2(1,\ell,\alpha)|$ is a subgroup of $\GL_2$, hence $\cdim G_2(1,\ell,\alpha) = 2$. In fact, if $\alpha(\bar 1) = \bar n$, then the subgroup of $\GL_2$ generated by $\left[\begin{smallmatrix} \zeta & \\ & \zeta^n \end{smallmatrix}\right]$ and $\left[\begin{smallmatrix} & 1 \\ 1 & \end{smallmatrix}\right]$ where $\zeta$ is a primitive $2^\ell$th  root of unity is isomorphic to  $G_2(1,\ell,\alpha)$. It follows that every non-commutative semi-direct product $\ZZ/2^k \ltimes \ZZ/2^\ell$ where $k>1$ has covariant dimension at least 3. This gives \cite[Proposition 6.2]{KS07}.

If $p\geq 3$, then $\cdim G_p(1,\ell,\alpha) \geq 3$ if $\alpha$ is non-trivial, because a faithful group of covariant dimension~2 admits a surjective homomorphism onto $D_{2n}$ ($n\geq 2$), $A_4$, $S_4$ or $A_5$ by \cite[Proposition~10.1]{KS07}. Thus a non-commutative semi-direct product  $\ZZ/p^k\ltimes\ZZ/p^\ell$ where $p\geq 3$ and $k\geq2$ has covariant dimension at least 4.
\end{exa}

\section{Faithful Groups}\label{gaschutz.section}
 Let $N_G \subset G$ denote  the subgroup generated by the minimal elements (under set inclusion) among the nontrivial normal abelian subgroups of $G$. Our work in \cite{KS07} used the following criterion of Gasch\"utz.
\begin{prop}[\cite{Ga54}]\label{prop-Gaschutz}
 Let $G$ be a finite group. Then $G$ is faithful if and only if $N_G$ is generated by the conjugacy class of one of its elements.
\end{prop} 
 We have the following corollary \cite[Corollary 4.1]{KS07}.
\begin{cor}\label{cor-Gaschutz}
Let $G$ be a non-faithful group and $H \subset G$ a subgroup containing $N_G$. Then $H$ is non-faithful, too.
\end{cor}

The proof given in \cite{KS07} claims that $N_G\subset N_H$. But this is false. For example, let $G=S_4\supset D_4$. Then $N_G$ is the Klein 4-group, while  $N_H=Z(D_4)\simeq\ZZ/2\ZZ$. Here is a correct proof.

\begin{lem} \label{gaschutz.lemma} Let $N_1,\dots,N_k$ be the minimal nontrivial normal abelian subgroups of $G$. Then
\begin{enumerate}
\item  Each $N_i$ is isomorphic to $(\ZZ/p\ZZ)^n$ for some $n\in \NN$ and prime $p$. 
\item  Let $L$ be a $G$-normal subgroup of $N_G$. There is a direct product $M$ of a subset of $\{N_1,\dots,N_k\}$ such that $N_G$ is the direct product $LM$.
\end{enumerate}
\end{lem}

\begin{proof}
By minimality, for any prime $p$ and $i$, $pN_i$ is zero or $N_i$. Thus $N_i\simeq(\ZZ/p\ZZ)^n$ for some $p$ and $n$ giving (1). For (2), inductively assume that we have found a $G$-normal subgroup $M_j$ of $N_G$ which is a direct product of a subset of $\{N_1,\dots,N_j\}$ such that $LM_j$ is a direct product containing $N_1,\dots,N_j$.  We start the induction with $M_0=\{e\}$.    If $LM_j$ contains $N_{j+1}$, then set $M_{j+1}=M_j$. If not, then $N_{j+1}\cap LM_j$ must be trivial, so that the products $M_{j+1}:=M_jN_{j+1}$ and $LM_{j+1}$ are direct where $N_{j+1}\subset LM_{j+1}$. Set $M=M_k$. Then $LM$ is a direct product containing all the generators of $N_G$, hence equals $N_G$.
\end{proof}

 \begin{cor}
 $N_G$ is a direct product of a subset of $\{N_1,\dots,N_k\}$, hence $N_G$ is abelian. 
  \end{cor}

\begin{proof}[Proof of Corollary \ref{cor-Gaschutz}]  
The subgroup $N_G \cap N_H \subset N_H$ is normal in $H$. By Lemma \ref{gaschutz.lemma}   it has a complement $M$. Now assume that $H$ is faithful. Then by Proposition \ref{prop-Gaschutz} there exists an element $(c,d) \in N_H = (N_G \cap N_H) \times M$ whose $H$-conjugacy class generates $N_H$. Then the $H$-conjugacy class of $c$ generates  $N_G \cap N_H$.  Now let $N_i$ be one of the minimal nontrivial normal abelian subgroups of $G$. By hypothesis, $N_i\subset H$, hence $N_i$ contains a minimal nontrivial   $H$-submodule $N'$. Then $N'\subset N_G\cap N_H$. The smallest $G$-stable subspace of $N_G$ containing $N'$ is $N_i$, hence $N_i$ lies in the $G$-submodule of $N_G$ generated by the conjugacy class of $c$. Since $N_i$ is arbitrary, we see that    $G$ is faithful.
\end{proof}

\begin{rem}\label{product.rem}
Let $G_1,G_2,\ldots,G_m$ be faithful groups. {\it Then the product $G_1\times\cdots\times G_m$ is faithful if and only if the orders of the centers $Z(G_i)$ are pairwise coprime.} In fact, the center of the product is cyclic if and only if the orders $|Z(G_i)|$ are pairwise coprime, and in this case the tensor product of irreducible faithful representations $V_i$ of $G_i$ is irreducible and faithful.
\end{rem}

\section{Groups of covariant dimension 2}\label{covdim2.section}
In \cite{KS07} we showed that a finite group  of covariant dimension $2$ is a subgroup of $\GL_{2}$ and thus admits a faithful $2$-dimensional representation. In particular, we have the following result (cf. \cite[Theorem 10.3]{KS07}).
\begin{thm}\label{nonfaithfulcov2.thm}
If $G$ is a non-faithful finite group of covariant dimension 2, then $G$ is abelian of rank 2.
\end{thm}
Unfortunately, there is a gap in the proof of Lemma 10.3 in \cite{KS07} which is used in the proof of the theorem. So we give a new proof here which avoids this lemma. We start with the following result.
\begin{lem}\label{subgroupPGL2.lem}
If $G$ is a non-commutative finite group of covariant dimension 2, then $G/Z(G)$ is isomorphic to a subgroup of $\PGL_{2}$.
\end{lem}
\begin{proof} We use the notation of section 3. Let $\phi\colon V \to W$ be a multihomogeneous minimal covariant of degree $A$. Set $X:=\overline{\phi(V)}\subset W$ and let $S$ denote the image of the homomorphism $T(A)\colon \Cst^{n} \to \Cst^{m}$. Since $S$ is non-trivial, Lemma~\ref{centerG.lem} shows that $\dim\PP(X) \leq 1$ and that $G/Z(G)$ acts faithfully on $\PP(X)$. Thus $\dim \PP(X) = 1$ and $G/Z(G)$ acts faithfully on the normalization $\PP^{1}$ of $\PP(X)$. The lemma follows.
\end{proof}
\begin{proof}[Proof of Theorem~\ref{nonfaithfulcov2.thm}]
Let $G$ be a minimal counterexample, i.e., $G$ is non-faithful and non-commutative of covariant dimension 2, and every strict subgroup is either commutative or faithful. By the lemma above, $G/Z(G)$ is isomorphic to $A_{5}$, $S_{4}$, $A_{4}$,  or $D_{2n}$, and the image of $N_{G}$ in $G/Z(G)$ is a normal abelian subgroup. 
\par\smallskip
\noindent
{\bf Claim 1:}
{\it There are no surjective homomorphisms from $G$ to $A_5$, $S_4$, or $A_4$.}
\par\smallskip
If $\rho$ is a surjective homomorphism from $G$ to $A_5$ then $\rho(N_G)$ is trivial. If $\rho$  is a surjective homomorphism from $G$ to $S_4$ then $\rho(N_G)\subset K$ where $K\subset S_4$ is the Klein 4-group. In both cases $\rho^{-1}(A_4)\subsetneq G$ is  
neither faithful (by Corollary~\ref{cor-Gaschutz}) nor commutative, contradicting the minimality assumption.

Now assume that there is a surjective homomorphism $\rho\colon G \to A_4$, and
let $g_3\in G$ be the preimage of an element of $A_4$ of order $3$. We  
may assume that the order of $g_3$ is a power $3^\ell$. 
Since $\rho(N_G) \subset K$, the strict subgroup  
$S:=\rho^{-1}(K)\subsetneq G$ is commutative. Denote by $S_2$ the 2-torsion of $S$. Since $\rho(S_2) = K$ we see that $S_2$ has rank 2. Moreover, $S_2$ is normalized by $g_3$, but not centralized, and so $\cdim \langle g_3,S_2 \rangle \geq 3$ by  \cite[Corollary~4.4]{KS07}. This contradiction proves Claim~1.

\smallskip
\noindent
{\bf Claim 2:} 
{\it For every prime $p>2$ the $p$-Sylow-subgroup $G_p\subset G$ is normal and commutative of rank $\leq 2$. Hence $G$ is a semidirect product $G_2\ltimes G'$  where $G':=\prod_{p>2}G_p$ and $G_2$ is a $2$-Sylow subgroup.}
\par\smallskip
From Claim~1 we know that $G/Z(G) \simeq D_{2n}$. Then Claim~2  follows, because every $p$-Sylow-subgroup of $D_{2n}$ for $p\neq 2$ is normal and cyclic.

\smallskip
Now we can finish the proof. The case that $G=G_2$   is handled in  \cite[Lemma~10.2]{KS07}, so  
we can assume that  $G'$ is non-trivial. If $G_2$ commutes with $G'$, then $G_2$ is non-commutative and faithful. Moreover, no $G_p$ can be of  rank 2, else we have a subgroup which is a product $H:=G_2\times(\ZZ/p)^2$, and we have $\cdim H\geq 3$ by Corollary~\ref{timesZp.cor}. 
So $G'$ has rank ~1. Then  $G'$ is cyclic, hence $G$ is faithful by Remark~\ref{product.rem},  which is a contradiction.  Hence we may assume that $G_2$  acts nontrivially on $G'$.

It is clear that $N_G=N_2\times N'$ where $N_2=N_G\cap G_2$ and $N':=N_G\cap G'$. Since $G_2$ acts
nontrivially on $G'$, there is a $g\in G_2$ which induces an order 2 automorphism of some $G_p\neq\{e\}$. Then
one can see that $g$ acts nontrivially on $N_{G_p}$. Since $G$ is not faithful, $N_G$ is not generated by a
conjugacy class (Proposition \ref{prop-Gaschutz}) and the same holds for the subgroupÊ $H:=\langle
g,N_2\rangle\ltimes N'$ (Corollary~\ref{cor-Gaschutz}). Thus $H$ is neither faithful nor commutative, so that
it must equal $G$ by minimality.  It follows that each nontrivial $G_p$, for $p\neq 2$, is isomorphic to either $\ZZ/p$ or 
$(\ZZ/p)^2$.

Suppose that $G_p=(\ZZ/p)^2$ for some $p$. If $g$ acts trivially on $G_p$, then it must act nontrivially on some $G_q$, and then we have the subgroup $(\langle g\rangle\ltimes G_q)\times (\ZZ/p)^2$ which by Corollary~\ref{timesZp.cor}(2)  has covariant dimension at least $3$. 
If $g$ acts by sending each element of $G_{p}$ to its inverse, then, by Corollary~\ref{timesZp.cor}(1) and Corollary~\ref{abelian.cor}, 
$$
\cdim \langle g\rangle\ltimes G_p = \cdim (\langle g\rangle\ltimes G_p)\times \ZZ/p \geq \cdim (\ZZ/p)^{3} = 3.
$$
So we can assume that  $g$ acts on $G_p$ fixing one generator and sending the other to its inverse for every $G_p$ of rank 2. Thus $G'$ is generated by the conjugacy class of a single element. It follows that $N_2$ must have rank 2 and  $g$ must commute with $N_2$, else $N_2\times G'$ is generated by the conjugacy class of a single element. Suppose that $\langle g\rangle\cap N_2\simeq\ZZ/2$. If $g$ acts nontrivially on $\ZZ/p\subset G'$, then $\langle g,N_2\rangle\ltimes\ZZ/p$ contains a subgroup $(\langle g\rangle\ltimes\ZZ/p)\times\ZZ/2$ which has covariant dimension 3, again by Corollary~\ref{timesZp.cor}(2). If $\langle g\rangle\cap N_2=\{e\}$, then we have the subgroup $(\langle g\rangle \ltimes\ZZ/p)\times (\ZZ/2)^2$ which has covariant dimension three by Corollary~\ref{timesZp.cor}(1). This finishes the proof of the theorem.
\end{proof}

\section{Errata to \cite{KS07}}\label {errata.section}

\noindent First sentence after Definition 4.1. Replace ``simple groups.'' by ``nonabelian simple groups.''

\noindent Proof of Proposition 4.3, second paragraph. Replace ``is divisible by $m$'' with ``is congruent to 1 mod $m$.''

\noindent Proof of Corollary 5.1 last sentence. Replace ``Corollary 4.3'' by ``Proposition 4.3.''

\noindent Proof of Proposition 6.1 second paragraph. Change `$`\phi|_W$'' to $F |_W$.''

\noindent Proof of Proposition 6.1 first displayed formula. Replace ``$F(w,t)$'' and ``$F_0(w,t)$'' by ``$F(w,t^m)$'' and ``$F_0(w,t^m)$.''

\noindent Top of page 282. Change ``trivial stabilizer'' to ``trivial stabilizer or stabilizer $\pm I$.''


\end{document}